\documentclass[12pt]{amsart}
\usepackage{amssymb} 

\newtheorem{thm}{Theorem}[section]
\newtheorem{cor}[thm]{Corollary}
\newtheorem{lemma}[thm]{Lemma}

\theoremstyle{definition}
\newtheorem{defn}[thm]{Definition}

\numberwithin{equation}{section}

\usepackage[colorlinks]{hyperref}

\begin{document}
\title[Finite solvable groups with ...  ]
{Finite solvable groups with a  nilpotent  normal complement subgroup}%
\author[    ]{  Mohsen Amiri}%
\address{ Departamento de  Matem\'{a}tica, Universidade Federal do Amazonas.}%
\email{m.amiri77@gmail.com}%
\email{}
\thanks{}
\thanks{}
 
\begin{abstract}
 Let $G$ be a finite solvable  group and $H$ a non-normal 
core-free     subgroup of $G$.  We show that   if   the normalizer of any non-trivial normal subgroup of  $Fit(H)$ is equal $H$, then $H$ has a nilpotent  normal  complement $K$ such that $G=KH$ and $KZ(Fit(H))$ is a Frobenius group.
\end{abstract}

\maketitle

\section{  Introduction }
A famous theorem of Burnside \cite{1} asserts that if a Sylow $p$-subgroup $P$ of
a finite $G$ lies in the centre of its normalizer, then $G$ is $p$-nilpotent, that is, $G$ has a
normal Hall $p'$-subgroup. Another well-known result due to Frobenius \cite{2} showed that a
finite group is $p$-nilpotent if and only if the normalizer of any non-trivial $p$-subgroup is $p$-nilpotent; or the quotient group of normalizer by centraliser of any non-trivial $p$-subgroup
of the group is a $p$-group. Also, Thompson (1964) showed that if $p$ is an odd prime and the groups $N_G(J(P))$ and $C_G(Z(P))$ both have normal $p$-complements for a Sylow $p$-subgroup of $G$, then $G$ has a normal $p$-complement.
Several mathematicians     have extended the above three theorems in different ways, for example see \cite{3,6}.
Let the group $G = KH$ be the product of two nilpotent subgroups $K$ and $H$.
If $G$ is finite, then it is
known that $G$ is soluble (Kegel and Wielandt \cite{12}) and that its Fitting length is bounded by the sum of the nilpotent classes of $K$ and $H$ (Gross \cite{15}) or in terms of the derived length of the factors (Parmeggiani \cite{18}, Pennington \cite{19}). 
 In this note,   we prove the following theorem:

\begin{thm}\label{comp22}
Let $G$ be a finite solvable  group and $H$ a non-normal    subgroup of $G$, and let $C:=Cor_G(H)\neq H$. If     the normalizer of any non-trivial  normal subgroup of  $Fit(\frac{H}{C})$ is equal $\frac{H}{C}$, then   $G=Fit(\frac{G}{C})\rtimes \frac{H}{C}$ and $Fit(\frac{G}{C})Z(Fit(\frac{H}{C}))$ is a Frobenius group.
\end{thm}

If $G=PSL(2,17)$ and  $H\in Syl_2(G)$, then 
$N_G(U)=H$ for any normal subgroup $U\neq 1$ of $H$. Hence,  Theorem \ref{comp22}, is not true for  
non-solvable groups. 
Also, we show that 
if $G$ is a non-solvable finite group and  $H\neq Cor_G(H)$ a   nilpotent subgroup of $G$ which is not a $2$-group, then for some    subgroup $U$ of $H$, we have  $H<N_G(U)<G$.   
\section{Results}

We start    with the following definition.
\begin{defn}\label{deff}
Let $G$ be a group, and let $H$ be a  proper subgroup of $G$. Let $C=Cor_G(H)$. 
We say $H$ is a  maximal normalizer whenever for any   normal subgroup $\frac{L}{C}\neq 1$ of $Fit(\frac{H}{C})$, we have $N_{\frac{G}{C}}(\frac{L}{C})=\frac{H}{C}$.
\end{defn}
 
The following lemma  shows that any nilpotent maximal normalizer is a Hall subgroup.

\begin{lemma}\label{hal}
Let $G$ be a finite group and   $H$ is a nilpotent maximal normalizer. Then
 
  (i) $\frac{H}{Cor_G(H)}$ is a Hall subgroup of $\frac{G}{Cor_G(H)}$.
  
  (ii) $\frac{H}{Cor_G(H)}$ is a maximal normalizer subgroup of $\frac{G}{Cor_G(H)}$.

\end{lemma}
\begin{proof}{ Let $C=Cor_G(H)$.

  (i) It is enough to show that every Sylow $p$-subgroup of $\frac{H}{C}$ is a Sylow $p$-subgroup of $\frac{G}{C}$. Suppose that $1\neq P\in H$ such that $P$ is not a normal subgroup of $G$ and $Q\in Syl_p(G)$ such that $P\subseteq Q$.  Since $N_G(H)$ is nilpotent, we have $N_G(H)\leq N_G(P)\leq N_G(Z(P))=H$ and so   $N_G(H)=N_G(Z(P))$.
 Now, if $P\neq Q$, then there exists $x\in N_Q(Z(P))\setminus P$ and so $x\in H$. On the other hand $N_Q(P)\leq N_G(Z(P))=H$,  which is a contradiction. Hence $P=Q$, as claimed.
 
 (ii) Let $1\neq \frac{L}{C}\leq Z(Fit(\frac{H}{C}))$ be a normal subgroup of $\frac{H}{C}$.
 Then $L\lhd H$. Since $L$ is not a normal subgroup of $G$, we have $N_G(L)=H$, by the definition of $H$.  
 Hence $N_G(\frac{L}{C})=\frac{H}{C}$. }
\end{proof}

An automorphism of a group is termed fixed-point-free if the only element of the group that is fixed under the action of the automorphism is the identity element. We need the following theorems related a group  with a fixed-point-free automorphism.
\begin{thm} (Burnside \cite{buu}) \label{burn}Let $\Phi$ be a fixed-point-free automorphism group on a group $G$.

(a) If the order of $\Phi$ is $pq$  for $p$ and $q$  not necessarily distinct primes, then $\Phi$
is cyclic.

(b) The Sylow $p$-subgroups of $\Phi$ are cyclic for $p>2$, they are cyclic or generalized quaternion
groups for $p=2$.
\end{thm}
\begin{thm}\label{thomm}(Thompson \cite{Rob} (10.5.4))
Let $G$ be a finite group and let $p$ be a prime. If $G$ has a
fixed-point-free automorphism $\alpha$ of order $p$, then $G$ is nilpotent.
\end{thm}

 \begin{thm}
     \label{comp}
 Let $G$ be a finite solvable group, and let $\frac{H}{Cor_G(H)}$ be a    maximal normalier    of $\frac{G}{Cor_G(H)}$.
   Then  $\frac{G}{Cor_G(H)}=Fit(\frac{G}{Cor_G(H)})\rtimes \frac{H}{Cor_G(H)}$ and $Fit(\frac{G}{Cor_G(H)}) Z(Fit(\frac{H}{Cor_G(H)})$ is a Frobeius group.

\end{thm}

\begin{proof} {Let $Fit(\frac{G}{Cor_G(H)})=\frac{K}{Cor_G(H)}$.
    We proceed by induction on $|G|$.  We may assume that $Cor_G(H)=1$. 
  Let $N$ be a normal minimal subgroup of $G$.
  First suppose that $NH=G$.
  Since $N$ is a normal minimal subgroup, $H\cap N=1$ and  $N=Fit(G)$, because $Cor_G(H)=1$.
  Let  $Y=NZ(Fit(H))$. For any $a\in Z(Fit(H))\setminus \{1\}$, we have  $N_{G}(\langle a\rangle)=H$, so $C_Y(a)=Z(Fit(H))$. Therefore  $Y$ is a Frobenius group. So suppose that $NH\neq G$.
  We claim that $\frac{HN}{N}$ is a core free maximal normalizer subgroup of $\frac{G}{N}$.
Suppose for a contradiction  that  
$Cor_{\frac{G}{N}}(\frac{H N}{N}):=\frac{LN}{N}\neq 1$ where $L\lhd H$. Since $Fit(L)$ is a characteristic subgroup of $L$, we deduce that
 $V:=Fit(L)\lhd H$, and so $V\lhd Fit(H)$.
Since   $\frac{VN}{N}= Fit(\frac{LN}{N})$, we have $NV\lhd G$.
Let $a\in V\cap Z(Fit(H))\setminus\{1\}$ of order prime number $p$. Since 
$N_G(\langle a\rangle)=H$, we deduce that  $Fit(G) \langle a\rangle$ is a Frobenius group.
So $$Fit(G)\leq \langle a^G\rangle\leq VL\lhd G.$$
It follows that $Fit(G)=N$ is a normal minimal subgroup of $G$.
Then $G=NM$ where $M$ is a maximal subgroup of $G$ such that $N\cap M=1$.
  We may assume that $H\leq M$. Then $V\lhd M$.
By our assumption 
$M=N_G(V)=H$, which is a contradiction.
 Therefore  $Cor_{\frac{G}{N}}(\frac{HN}{N})=1$.
Let $1\neq U\lhd Fit(H)$, and let  
$M:=N_G(UN)$.
Since $Cor_G(\frac{HN}{N})=1$, we have $M\neq G$.
If $L:=Cor_{M}(H)\neq 1$, then 
$F:=Fit(L)\neq 1$. Since $L\lhd M$, we have $F\lhd M$. It follows that 
$M=N_{M}(F)=H$, which is a contradiction.
So $L=1$.
By induction hypothesis, 
$M=Fit(M)H$. 
Let $u\in Z(Fit(H))\cap U$ of prime order.
Since $Fit(M)\langle u\rangle$ is a Frobenius group, we have 
$$Fit(M)\leq \langle u^{Fit(M)\langle u\rangle}\rangle\leq UN\lhd M,$$ and so $Fit(M)=N$.
Hence $N_{G}(UN)=M=HN$. Consequently, $N_{\frac{G}{N}}(\frac{UN}{N})=\frac{HN}{N}$.
It follows from  $Fit(\frac{HN}{N})\cong Fit(\frac{H}{N\cap H})\cong Fit(H)$
that  $\frac{HN}{N}$ is a maximal normalizer of $\frac{G}{N}$, as claimed.
By induction hypothesis, 
$\frac{G}{N}=Fit(\frac{G}{N})\rtimes \frac{HN}{N}$.
Let $Fit(\frac{G}{N})=\frac{V}{N}$.
  It follows from $V\cap H\leq N$ and $H\cap N=1$, that  $V\cap H=1$.
 Since  $VZ(Fit(H))$ is  a Frobenius group, $V$ is a nilpotent group,
 so $Fit(G)=V$.

}

     \end{proof}

Now, we show that if a finite group $G$ has a nilotent maximal normalizer which is not a Sylow $2$-subgroup, then $G$ is a solvable group.
\begin{thm}\label{rem23}
Let $G$ be a finite group and $H$ 
 a  nilpotent maximal normalizer of $G$. Then  $G$ is a solvable group or $H\in Syl_2(G)$.
\end{thm}
\begin{proof}
{ Let $G$ be   the smallest counterexample. From the minimality  of $G$, and Lemma \ref{hal}(ii) we may assume that $Cor_G(H)=1$. Then  $G$ is not a solvable group and  $H$ is not  a $2$-group. 
Let $P\in Syl_p(H)$ where $p>2$.
Since $N_G(J(P))=H$ and $C_G(Z(P))=H$, by Thompson $p$-complement theorem, $G=K\rtimes P$. If $K\cap H$ is not a Sylow $2$-subgroup, then by minimality of $G$, $K$ is a solvable group, therefore $G$ is a solvable group, which is a contradiction.
So suppose that $K\cap H$ is a Sylow $2$-subgroup. 
First suppose that $|P|>p$.
If $P=Z(P)$, let $M$ be a maximal subgroup of $P$, and if $P\neq Z(P)$, let $M=Z(P)$. Let $U=KM$, and let $R$ be a non-trivial normal subgroup of $M$.
Since $R\leq Z(H)$, we have 
$R\lhd H$, so
$N_G(R)=H$, and hence $$N_U(R)=U\cap N_G(R)=U\cap H=M.$$ Consequently, $M$ is a maximal normalizer of $U$.
By minimality of $G$, $U$ is a solvable group, and so $G$ is a solvable group, which is a contradiction. So suppose that $|P|=p$.
Then $P$ acts on $K$ by conjugation.
Since $C_G(P)=H$ is a nilpotent group, from the main result of \cite{Emerson}, $K$ is a nilpotent group, which is a contradiction. It follows that $G$ is a solvable group.

}

\end{proof}
   \begin{cor}
Let $G$ be a finite non-solvable group, and let $H\neq Cor_G(H)$ be a nilpotent subgroup of $G$ which is not a $2$-group. Then $H$ has a   subgroup $U$ such that $H<N_G(U)<G$.
   \end{cor}

By applying the result (II) of \cite{Ros}, we have the following theorem about non-solvable group with a nilpotent  maximal normalizer subgroup. 
    \begin{thm}\label{simp}
        Let $G$ be a finite non-solvable  group  and let $L=Fit(G)$. 
        If  $H$ is a
         nilpotent maximal normalizer of $G$, then  $L\leq H$ and $G$ has a unique minimal normal  subgroup $K$, 
        $K$ is a direct product of copies of a simple group $S\cong PSL(2,p)$ with dihedral Sylow $2$-subgroups, and $\frac{G}{K}$ is
a $2$-group.   
        
    \end{thm}
      \begin{proof}{
      Clearly, $H\nleq L$.
      From Theorem \ref{rem23},
      $H\in Syl_2(G)$.
      By Result (II) of \cite{Ros}, $\frac{G}{L}$ has a unique minimal normal subgroup $\frac{K}{L}$, $\frac{K}{L}$ is a direct product of copies of a simple group with dihedral Sylow $2$-subgroups, and $\frac{G}{K}$ is
a $2$-group.
      If $L\nleq H$, then $LH$ is a solvable group. From Theorem \ref{comp}, $LH$ is a Frobenius group with complement $H$.
By Theorem \ref{burn}, Sylow $2$-subgroup  of $H$ is cyclic or generalized quaternion group, which is a contradiction. Hence $L\leq H$.}
      \end{proof}
       
\section*{Acknowledgement}
The author would like to thank the referee for the useful comments and gratefully acknowledge the support of Coordenação de Aperfeiçoamento de Pessoal de Nível Superior (CAPES), Fundação de Apoio a Pesquisa do Distrito Federal(FAPDF).

 \end{document}